\documentclass[12pt]{amsart}
\usepackage{amsfonts,amsthm,amsmath,amscd,amssymb}
\usepackage[latin2]{inputenc}
\usepackage{t1enc}
\usepackage[mathscr]{eucal}
\usepackage{indentfirst}
\usepackage{graphicx}
\usepackage{graphics}
\usepackage{xcolor}
\usepackage{pict2e}
\usepackage{epic}
\usepackage{hyperref}
\numberwithin{equation}{section}
\usepackage[margin=2.9cm]{geometry}
\usepackage{epstopdf} 
\usepackage{cite}
\usepackage{parskip}
\usepackage{array}
\usepackage{booktabs}
\usepackage{enumerate}
\usepackage{mathtools}

\theoremstyle{plain}
\newtheorem{thm}{Theorem}[section]
\newtheorem{lem}[thm]{Lemma}
\newtheorem{cor}[thm]{Corollary}
\newtheorem{prop}[thm]{Proposition}

\newtheorem{ques}[thm]{Question}

 \theoremstyle{definition}
\newtheorem{defn}[thm]{Definition}
\newtheorem{rem}[thm]{Remark}
\newtheorem{ex}[thm]{Example}

\newcommand{\bm}[1]{\mathbf{#1}}

\newcommand{\mb}[1]{\mathbb{#1}}

\newcommand{\mf}[1]{\mathfrak{#1}}

\newcommand{\Char}{\operatorname{char}}

\newcommand{\rank}{\operatorname{rank}}
\newcommand{\Jac}{\operatorname{Jac}}
\newcommand{\Bez}{\operatorname{B\acute ez}}
\newcommand{\rBez}{\overline{\operatorname{B\acute ez}}}
\newcommand{\Mod}{\ \mathrm{mod}\ }

\newcommand{\pcoor}[1]{%
  \begingroup\lccode`~=`: \lowercase{\endgroup
  \edef~}{\mathbin{\mathchar\the\mathcode`:}\nobreak}%
  [
  \begingroup
  \mathcode`:=\string"8000
  #1%
  \endgroup 
  ]
}

\makeatletter
\@namedef{subjclassname@2020}{\textup{2020} Mathematics Subject Classification}
\makeatother

\begin{document}
\title{B\'ezoutians and injectivity of polynomial maps}

\author{Stephen McKean}

\address{Department of Mathematics\\
Harvard University\\
1 Oxford Street\\
Cambridge, MA 02138} 

\email{smckean@math.harvard.edu}
\urladdr{shmckean.github.io}


\begin{abstract}
We prove that an endomorphism $f$ of affine space is injective on rational points if its B\'ezoutian is constant. Similarly, $f$ is injective at a given rational point if its reduced B\'ezoutian is constant. We also show that if the Jacobian determinant of $f$ is invertible, then $f$ is injective at a given rational point if and only if its reduced B\'ezoutian is constant.
\end{abstract}

\maketitle

\section{Introduction}
Let $k$ be a field, and let $f=(f_1,\ldots,f_n)\colon\mb{A}^n_k\to \mb{A}^n_k$ be a polynomial morphism. In this note, we study the injectivity of $f$ at the origin using the \textit{multivariate B\'ezoutian}.

\begin{defn}\label{def:bezoutian}
Let $\bm{x}=(x_1,\ldots,x_n)$ and $\bm{y}=(y_1,\ldots,y_n)$. The \textit{(multivariate) B\'ezoutian} of $f:=(f_1(\bm{x}),\ldots,f_n(\bm{x}))$ is the determinant
\[\Bez(f):=\det(\Delta_{ij})\in k[\bm{x},\bm{y}],\]
where
\[\Delta_{ij}=\frac{f_i(y_1,\ldots,y_{j-1},x_j,\ldots,x_n)-f_i(y_1,\ldots,y_j,x_{j+1},\ldots,x_n)}{x_j-y_j}.\]
The \textit{reduced B\'ezoutian} of $f$ is $\rBez(f):=\Bez(f)\Mod(f(\bm{x}),f(\bm{y}))$.
\end{defn}

\begin{defn}
Let $R$ be a polynomial ring over a field $k$. Let $I$ be an ideal of $R$. If $I$ is a proper ideal, then $k\subseteq R/I$. We say that an element $c\in R/I$ is \textit{constant} if (i) $I$ is a proper ideal and $c\in k$, or if (ii) $I$ is not proper, in which case $c=0$.
\end{defn}

Multivariate B\'ezoutians generalize the classical B\'ezoutian of a univariate polynomial. They naturally arise in the study of global residues (see e.g.~\cite{SS75,BCRS96}). We will show that $f$ is injective at a $k$-rational point $q$ if $\rBez(f-q)$ is constant.

\begin{thm}\label{thm:main}
Let $k$ be a field, and let $f=(f_1,\ldots,f_n)\colon\mb{A}^n_k\to\mb{A}^n_k$ be a polynomial morphism with finite fibers. Let $q=(q_1,\ldots,q_n)$ be a $k$-rational point of $\mb{A}^n_k$. If $\rBez(f-q)$ is constant, then $|f^{-1}(q)|\leq 1$.
\end{thm}

We will also see that $f$ is injective on $k$-rational points if $\Bez(f)$ is constant.

\begin{cor}\label{cor:global inj}
If $\Bez(f)$ is constant, then $f$ is injective on $k$-rational points.
\end{cor}

In general, the constancy of $\Bez(f)$ (or $\rBez(f-q)$) is a sufficient but not necessary condition for injectivity (see Example~\ref{ex:non-constant but injective}). In Lemma~\ref{lem:converse}, we describe circumstances under which Theorem~\ref{thm:main} gives a necessary and sufficient condition for injectivity at a rational point.

Bass, Connell, and Wright have shown that if $k$ has characteristic 0 and $\Jac(f)\in k^\times$, then $f$ is invertible if and only if $f$ is injective on $k$-rational points~\cite[Theorem 2.1]{BCW82}. In particular, Theorem~\ref{thm:main} and Lemma~\ref{lem:converse} give a reformulation of the Jacobian conjecture in characteristic 0.

\begin{cor}\label{cor:jacobian}
Let $k$ be an algebraically closed field of characteristic 0, and let $f=(f_1,\ldots,f_n)\colon\mb{A}^n_k\to\mb{A}^n_k$ be a polynomial morphism. Assume $\Jac(f)\in k^\times$. Then $\rBez(f-q)$ is constant for all $q\in\mb{A}^n_k(k)$ if and only if $f$ is invertible.
\end{cor}

Note that $\Bez(f)=\Bez(f-q)$ for any $k$-rational point $q$. In particular, if $\Bez(f)$ is constant, then $\rBez(f-q)$ is constant for all $q\in\mb{A}^n_k(k)$. This gives a sufficient but not necessary criterion for the Jacobian conjecture.

\begin{cor}\label{cor:jacobian not necessary}
Let $k$ be a field of characteristic 0, and let $f=(f_1,\ldots,f_n)\colon\mb{A}^n_k\to\mb{A}^n_k$ be a polynomial morphism. Assume $\Jac(f)\in k^\times$. If $\Bez(f)$ is constant, then $f$ is invertible.
\end{cor}

The key observation leading to Theorem~\ref{thm:main} is that $\rBez(f-q)$ records information about the dimension of $k[\bm{x}]/(f-q)$ as a $k$-vector space. We will recall the relevant details about B\'ezoutians in Section~\ref{sec:bezoutians}. We will then prove Theorem~\ref{thm:main} in Section~\ref{sec:proof}. Finally, we discuss Theorem~\ref{thm:main} in the context of the Jacobian conjecture in Section~\ref{sec:discussion}.

\subsection*{Acknowledgements}
We thank Sri Iyengar for his support and feedback, as well as for helping the author catch a serious gap in the predecessor of this paper. We also thank Thomas Brazelton and Sabrina Pauli for various enlightening conversations about B\'ezoutians, and Cleto Miranda-Neto for helpful correspondence. The author received support from Kirsten Wickelgren's NSF CAREER grant (DMS-1552730).

\section{B\'ezoutians}\label{sec:bezoutians}
Throughout this section, let $f:\mb{A}^n_k\to\mb{A}^n_k$ be a morphism with finite fibers. This ensures that $(f_1,\ldots,f_n)$ is a complete intersection ideal, which allows us to utilize the multivariate B\'ezoutian~\cite[Section 3]{BCRS96}.

\begin{rem}\label{rem:bezoutian}
As noted by Scheja--Storch~\cite[p. 182]{SS75} and Becker--Cardinal--Roy--Szafraniec~\cite{BCRS96}, the B\'ezoutian records information about the dimension of $k[\bm{x}]/(f)$ as a $k$-vector space. To see this, consider the isomorphism
\[\mu\colon \frac{k[\bm{x}]}{(f)}\otimes_k \frac{k[\bm{x}]}{(f)}\to \frac{k[\bm{x},\bm{y}]}{(f(\bm{x}),f(\bm{y}))}\]
defined by $\mu(a(\bm{x})\otimes b(\bm{x}))=a(\bm{x})b(\bm{y})$. The inverse is characterized by $\mu^{-1}(x_i)=x_i\otimes 1$ and $\mu^{-1}(y_i)=1\otimes x_i$. Since $\mu$ is an isomorphism, there is an element $B\in k[\bm{x}]/(f)\otimes_k k[\bm{x}]/(f)$ such that $\mu(B)=\rBez(f)$. Moreover, given a basis $\{c_i\}$ for $k[\bm{x}]/(f)$, there exists a basis $\{d_i\}$ for $k[\bm{x}]/(f)$ such that $B=\sum_i c_i\otimes d_i$~\cite[Theorem 2.10(iii)]{BCRS96}.
\end{rem}

\begin{prop}\label{prop:constant bezoutian}
If $\rBez(f-q)$ is constant, then $\dim_k k[\bm{x}]/(f-q)\leq 1$.
\end{prop}
\begin{proof}
We will prove that if $\rBez(f)$ is constant, then $\dim_k k[\bm{x}]/(f)\leq 1$. The same proof holds after replacing $f$ with $f-q$. Given any basis $\{c_1,\ldots,c_m\}$ of $k[\bm{x}]/(f)$, write $\rBez(f)=\sum_{i,j}B_{ij}c_i(\bm{x})c_j(\bm{y})$, where $B_{ij}\in k$. The $m\times m$ matrix $(B_{ij})$ is non-singular by \cite[Theorem 1.2]{BMP21}, so $(B_{ij})$ must contain at least $m$ non-zero entries. In particular, the number of non-zero terms of $\sum_{i,j}B_{ij}c_i(\bm{x})c_j(\bm{y})$ is at least $m=\dim_k k[\bm{x}]/(f)$. 

First, suppose $\rBez(f)\in k^\times$. Pick a monomial basis $\{c_1,\ldots,c_m\}$ of $k[\bm{x}]/(f)$, so that $\sum_{i,j}B_{ij}c_i(\bm{x})c_j(\bm{y})$ consists of a single non-zero term, so $1\geq\dim_k k[\bm{x}]/(f)$ (and in fact, equality holds). Next, if $\rBez(f)=0$, then $m=\dim_k k[\bm{x}]/(f)=0$.
\end{proof}

Let $\Jac(f):=\det(\frac{\partial f_i}{\partial x_j})$ be the Jacobian of $f$, and let $\delta\colon k[\bm{x},\bm{y}]\to k[\bm{x}]$ be given by $\delta(a(\bm{x},\bm{y}))=a(\bm{x},\bm{x})$. We can recover $\Jac(f)$ from $\Bez(f)$. This appears in \cite[p. 184]{SS75} and, modulo $(f)$, in \cite[p. 90]{BCRS96}, but we recall the details here.

\begin{prop}\label{prop:bezoutian gives jacobian}
We have $\delta(\Bez(f))=\Jac(f)$.
\end{prop}
\begin{proof}
Note that $\delta$ is a ring homomorphism, so it suffices to show that $\delta(\Delta_{ij})=\frac{\partial f_i}{\partial x_j}$. The result follows by taking a formal partial derivative, as we now explain. Let
\[f_{ij}(\bm{x},y_j)=\frac{f_i(\bm{x})-f_i(x_1,\ldots,x_{j-1},y_j,x_{j+1},\ldots,x_n)}{x_j-y_j},\]
so that $\delta(\Delta_{ij})=f_{ij}(\bm{x},x_j)$. Since $\Delta_{ij}$ is a polynomial, $\delta(\Delta_{ij})$ and $f_{ij}$ are polynomials as well. Now
\begin{align*}
\frac{\partial f_i}{\partial x_j}&=\frac{\partial}{\partial x_j}\left(f_i(\bm{x})-f(x_1,\ldots,x_{j-1},y_j,x_{j+1},\ldots,x_n)\right)\\
&=\frac{\partial}{\partial x_j}\left(f_{ij}(\bm{x},y_j)\cdot(x_j-y_j)\right)\\
&=\frac{\partial f_{ij}}{\partial x_j}\cdot(x_j-y_j)+f_{ij}(\bm{x},y_j).
\end{align*}
Thus
\begin{align*}
\delta\left(\frac{\partial f_i}{\partial x_j}\right)&=\delta\left(\frac{\partial f_{ij}}{\partial x_j}\cdot(x_j-y_j)+f_{ij}(\bm{x},y_j)\right)\\
&=0+f_{ij}(\bm{x},x_j)\\
&=\delta(\Delta_{ij}).
\end{align*}
Since $\frac{\partial f_i}{\partial x_j}\in k[\bm{x}]$, we have $\delta(\frac{\partial f_i}{\partial x_j})=\frac{\partial f_i}{\partial x_j}$, which proves the desired result.
\end{proof}

\begin{rem}
If $\Jac(f)\in k^\times$ and $\rBez(f)$ is constant, then Proposition~\ref{prop:bezoutian gives jacobian} implies that $\rBez(f)\neq 0$.
\end{rem}

\begin{ex}
Let $f=(x_1^2,x_2^2,x_3^2)$. The set $\{1,x_1,x_2,x_3,x_1x_2,x_1x_3,x_2x_3,x_1x_2x_3\}$ is a basis for $k[\bm{x}]/(f)$. Let
\begin{align*}
B&=x_1x_2x_3\otimes 1+x_2x_3\otimes x_1+x_1x_3\otimes x_2+x_1x_2\otimes x_3\\
&+x_1\otimes x_2x_3+x_2\otimes x_1x_3+x_3\otimes x_1x_2+1\otimes x_1x_2x_3.
\end{align*}
By Definition~\ref{def:bezoutian}, we have
\begin{align*}
\Bez(f)&=x_1x_2x_3+x_2x_3y_1+x_1x_3y_2+x_1x_2y_3\\
&+x_1y_2y_3+x_2y_1y_3+x_3y_1y_2+y_1y_2y_3.
\end{align*}
One can readily check that $\mu(B)=\Bez(f)$. Moreover, $\delta(\Bez(f))=8x_1x_2x_3$, which is equal to $\Jac(f)$ (see Proposition~\ref{prop:bezoutian gives jacobian}).
\end{ex}

\section{Proof of Theorem~\ref{thm:main}}\label{sec:proof}
Let $q=(q_1,\ldots,q_n)\in\mb{A}^n_k$ be a $k$-rational point. As a consequence of the structure theorem for Artinian rings, the dimension of $k[\bm{x}]/(f-q)$ as a $k$-vector space is closely related to the fiber cardinality $|f^{-1}(q)|$.

\begin{prop}\label{prop:fiber dim}
Let $(f-q)=(f_1-q_1,\ldots,f_n-q_n)$ be an ideal in $k[\bm{x}]$. Suppose that $f^{-1}(q)=\{p_1,\ldots,p_m\}$ is a finite set of points. Then $\dim_k k[\bm{x}]/(f-q)\geq |f^{-1}(q)|$.
\end{prop}
\begin{proof}
Since $f^{-1}(q)=\mb{V}(f-q)$ is a finite set, $k[\bm{x}]/(f-q)$ is Artinian by \cite[\href{https://stacks.math.columbia.edu/tag/00KH}{Lemma 00KH}]{stacks}. Let $\mf{m}_i$ be the maximal ideal in $k[\bm{x}]$ corresponding to $p_i$. By the structure theorem for Artinian rings (see \cite[\href{https://stacks.math.columbia.edu/tag/00JA}{Lemma 00JA}]{stacks} or \cite[Theorem 8.7]{AM69}), there is an isomorphism
\[\frac{k[\bm{x}]}{(f-q)}\cong\prod_{i=1}^m\frac{k[\bm{x}]_{\mf{m}_i}}{(f-q)}.\]
Thus $\dim_k k[\bm{x}]/(f-q)=\sum_{i=1}^m\dim_k k[\bm{x}]_{\mf{m}_i}/(f-q)$, which implies the claim.
\end{proof}

We are now prepared to prove Theorem~\ref{thm:main} and Corollary~\ref{cor:global inj}.

\begin{proof}[Proof of Theorem~\ref{thm:main}]
By Proposition~\ref{prop:constant bezoutian}, we have $\dim_k k[\bm{x}]/(f-q)\leq 1$, so Proposition~\ref{prop:fiber dim} implies that $|f^{-1}(q)|\leq 1$.
\end{proof}

\begin{proof}[Proof of Corollary~\ref{cor:global inj}]
Note that for any $k$-rational point $q$, we have $\Bez(f-q)=\Bez(f)$. Thus if $\Bez(f)\in k$, then $\rBez(f-q)\in k$ for all $q\in\mb{A}^n_k(k)$. By Theorem~\ref{thm:main}, $f$ is injective on $k$-rational points.
\end{proof}

\section{Dru\.zkowski morphisms with constant B\'ezoutian}\label{sec:discussion}
If we assume that $\Jac(f)\in k$, then we get slightly stronger injectivity results. By the work of Bass, Connell, and Wright \cite[Theorem 2.1]{BCW82}, we can study the Jacobian conjecture by studying the injectivity of morphisms with $\Jac(f)\in k^\times$. We start with the following standard result.

\begin{prop}\label{prop:quasi finite}
If $f\colon\mb{A}^n_k\to\mb{A}^n_k$ has $\Jac(f)\in k^\times$, then $f$ has finite fibers.
\end{prop}
\begin{proof}
Let $X=k[\bm{x}]/(f)$, and recall that the module of K\"ahler differentials $\Omega_{X/k}$ is the cokernel of the Jacobian matrix $(\frac{\partial f_i}{\partial x_j})$. Since $\Jac(f)\in k^\times$, we have that $\Omega_{X/k}=0$ and hence $f$ is unramified. By \cite[\href{https://stacks.math.columbia.edu/tag/02V5}{Lemma 02V5}]{stacks}, $f$ is locally quasi-finite. Since $\mb{A}^n_k$ is Noetherian, $f\colon\mb{A}^n_k\to\mb{A}^n_k$ is quasi-compact. Thus~\cite[\href{https://stacks.math.columbia.edu/tag/01TJ}{Lemma 01TJ}]{stacks} implies that $f$ is quasi-finite. In particular, $f$ has finite fibers~\cite[\href{https://stacks.math.columbia.edu/tag/02NH}{Lemma 02NH}]{stacks}.
\end{proof}

\begin{rem}
The statement that $f$ has finite fibers is equivalent to $\mb{V}(f-q)$ being a finite set for all $q$. Assuming $\Jac(f)\in k^\times$, it was shown by van den Essen~\cite[Theorem 1.1.32]{vdE00} that $|\mb{V}(f-q)|\leq [k(\bm{x}):k(\bm{f})]$. Miranda-Neto also proved the finiteness of $\mb{V}(f-q)$ using derivations and differentials~\cite[Theorem 3.1]{MN19}.
\end{rem}

We saw in Proposition~\ref{prop:fiber dim} that $\dim_k k[\bm{x}]/(f-q)\geq |f^{-1}(q)|$. Assuming that $k$ is algebraically closed of characteristic 0 and $\Jac(f)\in k^\times$, this inequality is an equality.

\begin{prop}\label{prop:dim=fiber degree}
Let $k$ be an algebraically closed field of characteristic 0. If $f\colon\mb{A}^n_k\to\mb{A}^n_k$ has $\Jac(f)\in k^\times$, then $\dim_k k[\bm{x}]/(f-q)=|f^{-1}(q)|$. (See also~\cite[Corollary 3.2]{MN19}.)
\end{prop}
\begin{proof}
We need to show that if $p\in f^{-1}(q)$ with corresponding maximal ideal $\mf{m}$, then $\dim_k k[\bm{x}]_\mf{m}/(f-q)=1$. Since $f$ has finite fibers, $k[\bm{x}]_\mf{m}/(f-q)$ is a local Artin ring. In particular, $k[\bm{x}]_\mf{m}/(f-q)$ is a finitely generated algebra over its residue field. Moreover, the residue field is $k$, since $k$ is algebraically closed and $\mf{m}$ is a closed point. We will conclude by showing that $k[\bm{x}]_{\mf{m}}/(f-q)$ is in fact a field and noting that any finitely generated $k$-algebra is isomorphic to $k$ (since $k$ is algebraically closed).

Since $k$ has characteristic 0, \cite[(4.7) Korollar]{SS75} implies that $\Jac(f)$ generates the socle of $k[\bm{x}]_\mf{m}/(f-q)$, which is the annihilator of the maximal ideal $\mf{m}$. That is, the maximal ideal of $k[\bm{x}]_\mf{m}/(f-q)$ is annihilated by a scalar, so this maximal ideal must be the zero ideal. In particular, $k[\bm{x}]_\mf{m}/(f-q)$ is a field.

Alternatively, one can note that $\Jac(f)\in k^\times$ implies that $\mb{V}(f-q)$ is smooth as an affine scheme. Moreover, $\mb{V}(f-q)$ has Krull dimension zero by Proposition~\ref{prop:quasi finite}. Since $\Char{k}=0$, it follows that $\mb{V}(f-q)$ is regular, so $k[\bm{x}]_\mf{m}/(f-q)$ is a regular local ring over an algebraically closed field. By the Cohen structure theorem, the $\mf{m}$-adic completion of $k[\bm{x}]_\mf{m}/(f-q)$ is a ring of power series over $k$ in 0 generators (i.e. $k$ itself), so $k[\bm{x}]_\mf{m}/(f-q)\cong k$ as rings. But this suffices to prove that $\dim_k k[\bm{x}]_\mf{m}/(f-q)=1$.
\end{proof}

\begin{rem}
In the course of Proposition~\ref{prop:dim=fiber degree}, we have shown that if $k$ is a field of characteristic 0 and $f\colon\mb{A}^n_k\to\mb{A}^n_k$ has $\Jac(f)\in k^\times$, then the ideal $(f_1-q_1,\ldots,f_n-q_n)\subset k[\bm{x}]$ is radical. Indeed, the structure theorem for Artinian rings allows us to decompose $k[\bm{x}]/(f-q)$ as a product of local Artinian rings, each of which is a field by~\cite[(4.7) Korollar]{SS75}. In particular, $(f-q)$ is a finite intersection of maximal ideals and is hence radical. This gives a proof of \cite[Theorem 3.1]{MN19} not relying on derivations or differentials, as asked by Miranda-Neto~\cite[Remark 3.3]{MN19}.
\end{rem}

If $\Jac(f)\in k^\times$, we get a converse to Theorem~\ref{thm:main} (assuming $k$ is algebraically closed with $\Char{k}=0$).

\begin{lem}\label{lem:converse}
Let $k$ be an algebraically closed field of characteristic 0.  If $\Jac(f)\in k^\times$ and $f$ is injective at $q\in\mb{A}^n_k(k)$, then $\rBez(f-q)$ is constant.
\end{lem}
\begin{proof}
Since $f$ is injective at $q$, Proposition~\ref{prop:dim=fiber degree} implies that $\dim_k k[\bm{x}]/(f-q)\leq 1$. If $\dim_k k[\bm{x}]/(f-q)=0$, then $\rBez(f-q)=0\in k$. If $\dim_k k[\bm{x}]/(f-q)=1$, then $k[\bm{x},\bm{y}]/(f(\bm{x})-q,f(\bm{y})-q)\cong k\otimes_k k\cong k$. Thus $\rBez(f-q)\in k$, as desired.
\end{proof}

Corollary~\ref{cor:jacobian} follows from Theorem~\ref{thm:main}, Lemma~\ref{lem:converse}, and \cite[Theorem 2.1]{BCW82}. Using Corollary~\ref{cor:jacobian not necessary}, we can prove the Jacobian conjecture for any morphism whose B\'ezoutian is a constant. An important class of morphisms to consider are Dru\.zkowski morphisms.

\begin{defn}
A morphism $f\colon\mb{A}^n_k\to\mb{A}^n_k$ is called a \textit{Dru\.zkowski morphism} if $f$ is of the form $(x_1+(\sum_{i=1}^n a_{1i}x_i)^3,\ldots,x_n+(\sum_{i=1}^na_{ni}x_i)^3)$ with $\Jac(f)\in k^\times$. We also say that $f$ is the Dru\.zkowski morphism \textit{determined by} the matrix $(a_{ij})$.
\end{defn}

It was proved by Dru\.zkowski \cite[Theorem 3]{Dru83} that if the Jacobian conjecture is true for all Dru\.zkowski morphisms over a field $k$ of characteristic 0, then the Jacobian conjecture is true over $k$. By \cite[(1.1) Remark 4]{BCW82}, the Jacobian conjecture over $\mb{C}$ implies the Jacobian conjecture over all fields of characteristic 0.

If $(a_{ij})$ is strictly upper triangular or strictly lower triangular, then the Dru\.zkowski morphism determined by $(a_{ij})$ has constant B\'ezoutian. This allows us to recover the well-known solution of the Jacobian conjecture for such morphisms~\cite[Theorem 1.8]{Tru15}:

\begin{prop}\label{prop:strictly upper triangular}
Let $k$ be an algebraically closed field of characteristic 0. Suppose $a_{ij}=0$ either for all $i\geq j$ or for all $i\leq j$. Then the morphism 
\[\textstyle f:=(x_1+(\sum_{j=1}^n a_{1j}x_j)^3,\ldots,x_n+(\sum_{j=1}^n a_{nj}x_j)^3)\colon\mb{A}^n_k\to\mb{A}^n_k\]
is invertible and has $\Jac(f)=1$.
\end{prop}
\begin{proof}
First suppose $a_{ij}=0$ for all $i\geq j$. Since $a_{ij}=0$ for $i>j$, we have $\Delta_{ij}=0$ for $i>j$. Since $a_{ii}=0$ for all $i$, we have $\Delta_{ii}=1$ for all $i$. Thus $\Bez(f)=\Jac(f)=1$, and $f$ is invertible by Corollary~\ref{cor:jacobian not necessary}. Symmetrically, if $a_{ij}=0$ for all $i\leq j$, then we again have $\Bez(f)=\Jac(f)=1$.
\end{proof}

Proposition~\ref{prop:strictly upper triangular} follows from \cite[Theorem 5]{Dru83} when the rank of $(a_{ij})$ is 0, 1, 2, or $n-1$. As mentioned in \cite[Remark 6]{Dru83}, $f$ is a Dru\.zkowski morphism (in particular, $\Jac(f)\in k^\times$) only if $\rank(a_{ij})<n$. More strongly, Dru\.zkowski proved that if the Jacobian conjecture is true for all Dru\.zkowski morphisms with $(a_{ij})^2=0$, then the Jacobian conjecture is true in general \cite[Theorem 2]{Dru01}. 

Since every nilpotent matrix is similar to a strictly upper triangular matrix, and since an invertible matrix $P$ determines an automorphism $P\colon\mb{A}^n_k\to\mb{A}^n_k$ given by $P(\bm{x})=P\bm{x}^T$, Proposition~\ref{prop:strictly upper triangular} gives a potential to approach the Jacobian conjecture~\cite{Mei95}. Given an invertible matrix $P$ and a Dru\.zkowski morphism $f$ determined by $(a_{ij})$, we would like to find automorphisms $S,T\colon\mb{A}^n_k\to\mb{A}^n_k$ such that $S\circ f\circ T$ is the Dru\.zkowski morphism determined by $P(a_{ij})P^{-1}$. It suffices to reduce to the case where $P$ is an elementary matrix. If $P$ is a permutation matrix, then finding $S,T$ is straightforward.

\begin{prop}\cite[Proposition 3.1]{GTGZ99}\label{prop:permutation}
Let $P$ be a permutation matrix. If $f$ is the Dru\.zkowski morphism determined by $(a_{ij})$, then $P\circ f\circ P^{-1}$ is the Dru\.zkowski morphism determined by $P(a_{ij})P^{-1}$.
\end{prop}

Paired with Proposition~\ref{prop:strictly upper triangular}, we recover \cite[Theorem 3.2]{GTGZ99}. However, Meisters proved that not all nilpotent matrices are cubic similar to a strictly upper triangular matrix \cite{Mei95}, which suggests that finding such automorphisms $S,T$ is not trivial. Indeed, the obvious trick does not quite work when $P$ is a row multiplication matrix, as we show in Remark~\ref{rem:row mult}. Row addition matrices seem to be even more problematic than row multiplication matrices.

\begin{rem}\label{rem:row mult}
Let $\bm{e}_i$ be the $i^\text{th}$ standard column vector. Let $D_i=(\bm{e}_1\ \cdots\ m\bm{e}_i\ \cdots\ \bm{e}_n)$ for some $m\in k^\times$. If $f$ is the Dru\.zkowski morphism determined by $(a_{ij})$, then $D_i^3\circ f\circ D_i^{-1}$ has matrix $D_i(a_{ij})D_i^{-1}$ but is not quite a Dru\.zkowski morphism. Indeed, we compute
\begin{align*}
D_i^3\circ f\circ D_i^{-1}&=(\textstyle x_1+(m^{-1}a_{1i}x_i+\sum_{j\neq i}a_{1j}x_j)^3,\ldots,\\
&\qquad\textstyle m^2x_i+(a_{ii}x_i+\sum_{j\neq i}ma_{ij}x_j)^3,\ldots,\\
&\qquad\textstyle x_n+(m^{-1}a_{ni}x_i+\sum_{j\neq i}a_{nj}x_j)^3).
\end{align*}
This is a Dru\.zkowski morphism if and only if $m=\pm 1$, since we need $m^2x_i=x_i$.
\end{rem}

By Propositions~\ref{prop:strictly upper triangular} and~\ref{prop:permutation} and Remark~\ref{rem:row mult}, we get the following corollary.

\begin{cor}\label{cor:conjugate}
If $(a_{ij})$ is conjugate to a strictly upper triangular matrix via permutation matrices and $\pm 1$ row multiplications, then the Dru\.zkowski morphism determined by $(a_{ij})$ is invertible.
\end{cor}

\begin{ex}\label{ex:non-constant but injective}
If all entries of $(a_{ij})$ are non-negative real numbers, then $(a_{ij})$ is nilpotent if and only if it is permutation-similar to a strictly upper triangular matrix. Unfortunately, Corollary~\ref{cor:conjugate} does not cover all Dru\.zkowski morphisms. For example, the matrix
\[A=\begin{pmatrix}0&1&0\\-1&0&1\\0&1&0\end{pmatrix}\]
is not conjugate to a strictly upper triangular matrix via permutation matrices and $\pm 1$ row multiplications \cite[(1.6)]{GTGZ99}. Moreover, the Dru\.zkowski morphism $f_A$ determined by $A$ has $\Bez(f_A)\not\in k$, so one cannot hope for Corollary~\ref{cor:jacobian not necessary} to be of use for general Dru\.zkowski morphisms. Nevertheless, the Jacobian conjecture is true for $f_A$, so $\rBez(f_A-q)=1$ for all $q$.
\end{ex}

It is known that Dru\.zkowski morphisms are injective at the origin \cite[Proposition 1]{Dru83}, so $\rBez(f)=1$ for any Dru\.zkowski morphism $f$ by Lemma~\ref{lem:converse}. Indeed, $\Jac(f)=1$ for any Dru\.zkowski morphism, so $1\equiv\Jac(f)\Mod(f)$. Since $\rBez(f)\in k$ and $\delta$ is injective on $k$ (see Section~\ref{sec:bezoutians}), we have $\rBez(f)=1$. Similarly, $\Jac(f-q)=1$ for any $k$-rational point $q\in\mb{A}^n_k$, so if $\rBez(f-q)$ is constant, then $\rBez(f-q)=1$.

\begin{ques}\label{ques}
Given a Dru\.zkowski morphism $f$, is $\rBez(f-q)=1$ for all $q\in\mb{A}^n_k(k)$?
\end{ques}

Since $\rBez(f)=1$, we have $\Bez(f)\equiv 1\Mod(f(\bm{x}),f(\bm{y}))$. Since $\Bez(f)=\Bez(f-q)$ for any $q$, Question~\ref{ques} is asking if $\Bez(f)\equiv 1\Mod(f(\bm{x})-q,f(\bm{y})-q)$ for all $q$. 

We include a basic Sage script \cite{code} that takes as input a matrix $A$ and a $k$-rational point $q\in\mb{A}^n_k$ and returns as output $\Jac(f_A-q)$ and $\rBez(f_A-q)$ of the corresponding Dru\.zkowski morphism $f_A$. This script is based off of code written jointly with Thomas Brazelton and Sabrina Pauli for computing the $\mb{A}^1$-degree via the B\'ezoutian \cite{BMPcode}.

\bibliography{bezoutian-jacobian}{}
\bibliographystyle{alpha}
\end{document}